\documentclass[12pt]{amsart}
\usepackage{a4wide,enumerate,color}
\usepackage{amsmath}
\usepackage{amssymb}
\usepackage{graphicx,tikz}
\allowdisplaybreaks

\let\pa\partial
\let\na\nabla
\let\eps\varepsilon

\newcommand{\R}{{\mathbb R}}
\newcommand{\Z}{{\mathbb Z}}
\newcommand{\diver}{\operatorname{div}}
\newcommand{\T}{{\mathbb T}}

\newtheorem{theorem}{Theorem}
\newtheorem{lemma}[theorem]{Lemma}

\newtheorem{proposition}[theorem]{Proposition}
\newtheorem{remark}[theorem]{Remark}

\begin{document}

\title[Entropy-dissipating finite-difference schemes]{Entropy-dissipating 
finite-difference schemes for nonlinear fourth-order parabolic equations}

\author[M. Braukhoff]{Marcel Braukhoff}
\address{Institute for Analysis and Scientific Computing, Vienna University of
	Technology, Wiedner Hauptstra\ss e 8--10, 1040 Wien, Austria}
\email{marcel.braukhoff@tuwien.ac.at}

\author[A. J\"ungel]{Ansgar J\"ungel}
\address{Institute for Analysis and Scientific Computing, Vienna University of
	Technology, Wiedner Hauptstra\ss e 8--10, 1040 Wien, Austria}
\email{juengel@tuwien.ac.at}

\date{\today}

\thanks{The authors acknowledge partial support from
the Austrian Science Fund (FWF), grants F65, P30000, P33010, and W1245.}

\begin{abstract}
Structure-preserving
finite-difference schemes for general nonlinear fourth-order parabolic equations
on the one-dimensional torus are derived. 
Examples include the thin-film and the Derrida--Lebowitz--Speer--Spohn equations.
The schemes conserve the mass and dissipate the entropy. The scheme associated
to the logarithmic entropy also preserves the positivity. 
The idea of the derivation is to reformulate the equations in such a way that 
the chain rule is avoided. 
A central finite-difference discretization is then applied to the
reformulation. In this way, the same dissipation rates as in the continuous case
are recovered. The strategy can be extended to a multi-dimensional thin-film
equation. Numerical examples in one and two space dimensions illustrate the 
dissipation properties.
\end{abstract}



\maketitle


\section{Introduction}

The design of numerical schemes that preserve the structure of the associated
partial differential equations is an important task in numerical mathematics. 
In this paper, we develop new
finite-difference approximations conserving the mass, preserving the positivity,
and dissipating the entropy of nonlinear fourth-order parabolic equations
of the type
\begin{equation}\label{1.eq}
  \pa_t u = -J_x, \quad J = u^\beta u_{xxx}+au^{\beta-1}u_{xx}u_x + bu^{\beta-2}u_x^3
	\quad\mbox{in }\T,
\end{equation}
where $\T=\R/\Z$ is the one-dimensional torus, $a$, $b\in\R$, and $\beta\ge 0$,
together with the initial condition $u(0)=u^0$ in $\T$. We also discuss briefly
the discretization of multi-dimensional equations; see Section \ref{sec.tfe}. 

A special case of \eqref{1.eq} (with $a=b=0$) is the thin-film equation 
\begin{equation}\label{1.tfe}
  \pa_t u = -(u^\beta u_{xxx})_x, 
\end{equation}
which models the flow of a thin liquid along a solid surface with film height $u(x,t)$
or the thin neck of a Hele-Shaw flow in the lubrication approximation \cite{Ber95}. 
The multi-dimensional version is given by $\pa_t u = -\diver(u^\beta\na\Delta u)$
and its discretization will be discussed in Section \ref{sec.tfe}. 
Another example (with $a=-2$, $b=1$,
$\beta=0$) is the Derrida--Lebowitz--Speer--Spohn (DLSS) equation 
\begin{equation}\label{1.dlss}
  \pa_t u = -2\bigg(u\bigg(\frac{(\sqrt{u})_{xx}}{\sqrt{u}}\bigg)_x\bigg)_x
	= -\bigg(u_{xxx}-2\frac{u_{xx}u_x}{u} + \frac{u_x^3}{u^2}\bigg)_x,
\end{equation}
which arises as a scaling limit of
interface fluctuations in spin systems \cite{DLSS91} and describes the evolution
of the electron density $u(x,t)$ in a quantum semiconductor under simplifying
assumptions.

Equations \eqref{1.tfe} and \eqref{1.dlss} conserve the mass
$\int_\T u(x,t)dx$ and preserve the positivity of the solution, even for
the corresponding multi-dimensional versions \cite{DGG98,JuMa08}. They also dissipate
the entropy\footnote{Strictly speaking, equation \eqref{1.eq}
{\em produces} entropy and {\em dissipates} energy, but the notion
{\em entropy dissipation} seems to be common in numerical schemes.} 
in the sense that there exists a Lyapunov functional 
$S(t)=\int_\T s(u(x,t))dx$ with entropy density $s(u)$ such that the entropy production
$-dS/dt$ provides gradient estimates. 
For instance, \eqref{1.tfe} and \eqref{1.dlss} dissipate the entropy density
$s(u)=s_\alpha(u)$, where
\begin{align}
  s_\alpha(u) = \frac{u^\alpha}{\alpha(\alpha-1)} &\quad\mbox{for }\alpha>0,\
	\alpha\neq 0,1, \nonumber \\
  s_0(u) = -\log u &\quad\mbox{for }\alpha=0, \label{1.ent} \\
	s_1(u) = u(\log u-1) &\quad\mbox{for }\alpha=1, \nonumber
\end{align}
if $\frac32\le\alpha+\beta\le 3$ (thin-film equation) and
$0\le\alpha\le\frac32$ (DLSS equation); see \cite[Section 4]{JuMa06}.
These bounds are optimal.
We call $s_1(u)$ the Shannon entropy (density), $s_0(u)$ the logarithmic entropy,
and $s_\alpha(u)$ for $\alpha\neq 0,1$ a R\'enyi entropy. It holds that
$s_\alpha(u)\to s_0(u)$ pointwise for $\alpha\to 0$ and $s_\alpha(u)\to u\log u$
pointwise for $\alpha\to 1$. We prefer to define $s_1(u)$ as in \eqref{1.ent}
to avoid the additional term in $d(u\log u)/du=\log u+1$ which would complicate
the subsequent computations.
 
The proof of the dissipation of the entropy is based on suitable integrations by 
parts and the chain rule \cite{JuMa06}. 
On the discrete level, we face the problem that the chain rule is generally not
available. We are aware of two general strategies.

The first strategy is to exploit the gradient-flow structure of the
parabolic equation (if it exists). It involves only one integration by
parts, and the discrete chain rule can be formulated by means of suitable
mean functions. This idea was elaborated as the
Discrete Variational Derivative Method for finite-difference approximations 
\cite{FuMa10}. The gradient-flow formulation (with respect to the $L^2$-Wasserstein
metric) yields a natural semi-discretization in time of the evolution using the
minimizing movement scheme in finite-dimensional spaces from finite-volume
or finite-difference approximations. 
These techniques were used in \cite{DMM10,HuLi19,MaMa16,MaOs17} for
the multi-dimensional DLSS equation. It allows for the proof of entropy dissipation
of the Shannon entropy and the Fisher information $\int_\T(\sqrt{u})_x^2 dx$,
but not for general R\'enyi entropies, since no Wasserstein gradient-flow formulation
seems to be available for these functionals. The thin-film equation
with $0<\beta<1$ is shown in \cite{LMS12} to be a gradient flow with respect to a 
weighted Wasserstein metric. In the work \cite{ZhBe00}, a finite-difference scheme 
that dissipates the discrete $H^1$ norm of the solution to the one-dimensional
thin-film equation was analyzed.

The minimizing movement scheme is based on the implicit Euler method. We mention
that higher-order time discretizations were investigated too, in the framework
of semi-discrete problems; see \cite{BEJ14} using the two-step BDF method and
\cite{JuMi15} using one-leg multistep generalizations. 
A generic framework for Galerkin methods in space and discontinuous Galerkin methods 
in time was presented in \cite{Egg19}.

The second strategy uses time-continuous Markov chains on finite state spaces.
Birth-death processes that define the Markov chain can be interpreted as a 
finite-volume discretization of a one-dimensional Fokker--Planck equation,
and the dissipation of the discrete Shannon entropy can be proved.
The nonlinear integrations by parts are reduced to a discrete Bochner-type
inequality \cite{CDP09,FaMa15,JuYu17}, which is obtained by identifying the 
Radon--Nikodym derivative of a measure involving the jump rates of the Markov chain 
\cite[Section 2]{BCDP06}. It seems that this idea is restricted to linear
diffusion equations.

In this paper, we suggest a third strategy. The idea is to write the flux $J$
as a combination of derivatives of the function $s'(u)$. 
This allows for integrations by parts that can be extended to
the discrete level and it avoids the application of the chain rule. More precisely,
we determine two functions $A$ and $B$ depending on $v:=s'(u)$ and 
its derivatives such that $J=A_x-vB_x$.
The function $s'(u)$ is known in thermodynamics as the chemical potential,
and the formulation of the flux in terms of the chemical potential seems to be 
natural from a thermodynamic viewpoint. We apply this idea to fourth-order parabolic
equations for the first time. It turns out that for $s=s_\alpha$
with $\alpha\neq 1$, we can write
$$
  A = \frac{u^{\alpha+\beta}}{(\alpha-1)^2 v}(\lambda_1\xi_2+\lambda_2\xi_1^2), \quad
  B = \frac{u^{\alpha+\beta}}{(\alpha-1)^2 v^2}(\lambda_3\xi_2+\lambda_4\xi_1^2),
$$
where $\xi_1=v_x/v$, $\xi_2=v_{xx}/v$, and the coefficients $\lambda_i$ depend
on $a$, $b$, $\alpha$, and $\beta$; see \eqref{2.lam} and \eqref{2.lam0}.
Integrating by parts twice and using equation \eqref{1.eq} gives for
$S_\alpha=\int_\T s_\alpha(u)dx$:
\begin{equation*}
  \frac{dS_\alpha}{dt} = \int_\T Jv_x dx
	= -\int_\T \big(v_{xx}A - (vv_{x})_xB\big)dx.
\end{equation*}
The task is to show that the integrand, written as a polynomial in 
$(\xi_1,\xi_2)$, is nonnegative for all values of $(\xi_1,\xi_2)\in\R^2$. 
It follows from the product rule $(vv_x)_x=vv_{xx}+v_x^2$ that, for $\alpha\neq 1$,
\begin{align}
  \frac{dS_\alpha}{dt} &= -\int_\T(\xi_2 vA - (\xi_2+\xi_1^2)v^2 B)dx \nonumber \\
  &  = -\int_\T \frac{u^{\alpha+\beta}}{(\alpha-1)^2}\big((\lambda_1-\lambda_3)\xi_2^2
  + (\lambda_2-\lambda_3-\lambda_4)\xi_2\xi_1^2 - \lambda_4\xi_1^4\big)dx.
  \label{1.ei}
\end{align}
Under certain conditions on the parameters, we expect that the integrand
is bounded from below by $u^{\alpha+\beta}(\xi_2^2+\xi_1^4)$, up to a factor, 
which yields some gradient estimates.

On the discrete level, we imitate this idea: The flux $J=A_x-vB_x$ and the variables
$\xi_1$, $\xi_2$ of the polynomials $A$ and $B$ is discretized by
central finite differences. For this, let $\T_h=\T/(h\Z)$ be a discrete
torus with grid size $h>0$ and define the scheme
\begin{equation}\label{1.d1}
  \pa_t u_i = -\frac{1}{h}(J_{i+1/2}-J_{i-1/2}), \quad
	J_{i+1/2} = \frac{1}{h}(A_{i+1}-A_i) - \frac{1}{2h}(v_{i+1}+v_i)(B_{i+1}-B_i),
\end{equation}
with the initial condition $u_i(0)=u^0(i)$ for $i\in\T_h$, where $u_i=u(i)$, 
$v_i=s'(u_i)$, and $A_i$ and $B_i$ are the polynomials $A$ and $B$, evaluated at 
$i\in\T_h$, respectively; see \eqref{3.AB}. 
We show below that with the discrete entropy $S_\alpha^h=h\sum_{i\in\T_h}
s_\alpha(u)$, the discrete analog of \eqref{1.ei} becomes
$$
  \frac{dS_\alpha^h}{dt}
	= -h\sum_{i\in\T_h}\big(\xi_{2,i}v_iA_i - (v_i\xi_{2,i}+\xi_{1,i}^2)v_i^2B_i\big),
$$
where $\xi_{1,i}$, $\xi_{2,i}$ approximate $v_x(i)/v(i)$, $v_{xx}(i)/v(i)$,
respectively. This yields exactly the polynomial of the continuous case. Thus,
we obtain the same conditions on the parameters as for the continuous equation.

We still need a discrete analog of the product rule
 $(vv_x)_x=vv_{xx}+v_x^2$ to conclude. This is done by carefully choosing 
$\xi_{1,i}$ and $\xi_{2,i}$. Definition \eqref{3.xi21} ensures that
$v_i^2(\xi_{2,i}+\xi_{1,i}^2)= (v_{i+1}^2-2v_i+v_{i-1}^2)/(2h^2)$ which approximates
$\frac12(v^2)_{xx}=vv_{xx}+v_x^2$. This choice is used in the central scheme
\eqref{1.d1} for $J_{i+1/2}$. Noncentral schemes require different definitions
of $\xi_{1,i}$ and $\xi_{2,i}$; see Remark \ref{rem.xi}.

A drawback of our technique is that scheme \eqref{1.d1} depends on the entropy
to be dissipated. The scheme does not dissipate all admissible entropies.
In applications, however, one usually wants to dissipate only that entropy
which is physically relevant. 

Our main results can be sketched as follows:
\begin{itemize}
\item Lyapunov functional: Let $\alpha\ge 0$, $\alpha\neq 1$ and assume that
\begin{equation}\label{1.cond}
  K(\alpha,\beta) := -2\alpha^2+(3a-4\beta+9)\alpha-2\beta^2+(3a+9)\beta-9(a+b+1) 
	\ge 0.
\end{equation}
Then the continuous entropy $S_\alpha$ {\em and} the discrete entropy $S_\alpha^h$
are dissipated in the sense that $dS_\alpha/dt\le 0$ and $dS_\alpha^h/dt\le 0$,
i.e., $S_\alpha$ and $S_\alpha^h$ are Lyapunov functionals for $u(t)$ and $u_i(t)$,
respectively; see Theorems \ref{thm.ent} and \ref{thm.ent2}. Condition
\eqref{1.cond} is optimal for the thin-film and DLSS equations.
\item Entropy dissipation: 
Let $\alpha\ge 0$, $\alpha\neq 1$ and assume that $K(\alpha,\beta)>0$. Then
there exists a constant $c(\alpha,\beta)>0$ such that for all $t>0$,
\begin{align}
  \frac{dS_\alpha}{dt} + c(\alpha,\beta)\int_\T u^{\alpha+\beta}
	(\xi_2^2+\xi_1^4)dx &\le 0, \label{1.cei} \\
  \frac{dS_\alpha^h}{dt} + c(\alpha,\beta)
	h\sum_{i\in\T_h}\bar u_i^{\alpha+\beta}
	(\xi_{2,i}^2+\xi_{1,i}^4) &\le 0, \label{1.dei}
\end{align}
where $\xi_1=v_x/v$, $\xi_2=v_{xx}/v$, $v=s_\alpha(u)$, 
$\bar u_i$ is an arbitrary average of $(u_i)$, and $\xi_{1,i}$, 
$\xi_{2,i}$ are defined in \eqref{3.xi21}. This result is proved in
Theorems \ref{thm.ent} and \ref{thm.ent2}.

\item Case $\alpha=1$: For the case $\alpha=1$, we need the formulation
$J=A_x-wB_x$ instead of $J=A_x-vB_x$, where $w=s_0'(u)=-u^{-1}$ and $v=\log u$,
since $J$ generally does not depend on the logarithm. For details, see
Proposition \ref{prop.ent1}.

\item Case $\alpha=0$: We show that scheme \eqref{1.d1}
with $\alpha=0$ possesses global positive solutions. This result is a consequence
of the discrete entropy inequality and mass conservation, which imply that
$h\sum_{i\in\T_h}(u_i(t)-\log u_i(t))$ is bounded for all $t>0$. Consequently, 
$u_i(t)-\log u_i(t)$ is bounded for all $i\in\T_h$ and $t>0$, and since the
function $s\mapsto s-\log s$ tends to infinity if either $s\to 0$ or $s\to\infty$,
this proves that $u_i(t)$ is bounded from below and above. We refer to 
Proposition \ref{prop.ex} for details.

\item Multi-dimensional case: In principle, the multi-dimensional case can be treated
using functions $A$ and $B$ with many variables. Practically, however, the
computations are becoming too involved and it may be unclear how to
discretize mixed derivatives. One idea to overcome this issue is
to use scalar variables only, like $\xi_1=|\na v|/v$ and $\xi_2=\Delta v/v$.
This allows us to treat the multi-dimensional thin-film equation; see
Proposition \ref{prop.tfe}.
\end{itemize}

The paper is organized as follows. We prove the continuous entropy inequality
\eqref{1.cei} in Section \ref{sec.cont} and the discrete entropy inequality
\eqref{1.dei} in Section \ref{sec.disc}. A scheme for the multi-dimensional
thin-film equation is proposed and analyzed in Section \ref{sec.tfe}.
Numerical simulations for the thin-film and DLSS equations in one space
dimension and for a thin-film equation in two space dimensions
are presented in Section \ref{sec.sim}.


\section{General continuous equation}\label{sec.cont}

To prepare the discretization, we need to analyze the entropy dissipation properties
of the continuous equation \eqref{1.eq}. We show first that $J$ can be written
as $A_x-vB_x$ with $v=s_\alpha'(u)$ and functions $A$ and $B$ which depend
on $v$, $v_x$, and $v_{xx}$.

\begin{lemma}
Let $J$ be given as in \eqref{1.eq} and $s_\alpha$ as in \eqref{1.ent}
and let $\alpha\ge 0$ satisfy $\alpha\neq 1$. Then $J = A_x - vB_x$, where
\begin{align*}
  A &= u^{\alpha+\beta}\big(\lambda_1u^{2-2\alpha}(s_\alpha'(u))_{xx}
	+ \lambda_2(\alpha-1)u^{3-3\alpha}(s_\alpha'(u))_x^2\big) \nonumber \\
	&= \frac{u^{\alpha+\beta}}{(\alpha-1)^2v}\bigg(\lambda_1\frac{v_{xx}}{v}
	+ \lambda_2\bigg(\frac{v_x}{v}\bigg)^2\bigg), \nonumber \\
	B &= u^{\alpha+\beta}\big(\lambda_3(\alpha-1)u^{3-3\alpha}(s_\alpha'(u))_{xx}
	+ \lambda_4(\alpha-1)^2u^{4-4\alpha}(s_\alpha'(u))_x^2\big) \nonumber \\
	&= \frac{u^{\alpha+\beta}}{(\alpha-1)^2v^2}\bigg(\lambda_3\frac{v_{xx}}{v}
	+ \lambda_4\bigg(\frac{v_x}{v}\bigg)^2\bigg) \nonumber
\end{align*}
and 
\begin{align}
  \lambda_1 &= \frac{-2\alpha^2+(\beta+5)\alpha+a\beta-\beta^2-a-2b-3}{(\alpha-1)
	(\beta-2\alpha+3)} + \frac{2(\alpha-1)}{\beta-2\alpha+3}\lambda_4, \nonumber \\
	\lambda_2 &= \frac{2\alpha^2-(a+7)\alpha+2a+b+6}{(\alpha-1)(\beta-2\alpha+3)}
	+ \frac{\beta-3\alpha+4}{\beta-2\alpha+3}\lambda_4, \label{2.lam} \\
	\lambda_3 &= \frac{(a+1)\beta-\beta^2-a-2b}{(1-\alpha)(\beta-2\alpha+3)}
	+ \frac{2(\alpha-1)}{\beta-2\alpha+3}\lambda_4, \nonumber
\end{align}
with $\lambda_4\in\R$ being a free parameter.
\end{lemma}

The lemma shows in particular that the formulation $J=A_x-vB_x$ is not unique.
This fact is used to optimize later the range of admissible parameters
$\alpha$, $\beta$, $a$, and $b$.

\begin{proof}
Let $\alpha>0$ with $\alpha\neq 1$. A direct computation yields
\begin{align*}
  A &= \lambda_1 u^\beta u_{xx} + \big((\lambda_1+\lambda_2)\alpha-2\lambda_1-\lambda_2
	\big)u^{\beta-1}u_x^2, \\
  B &= \lambda_3(\alpha-1)u^{\beta-\alpha+1}u_{xx} + \big((\lambda_3+\lambda_4)\alpha
	-2\lambda_3-\lambda_4\big)u^{\beta-\alpha}u_x^2.
\end{align*}
Inserting these expressions into $A_x-vB_x$ and simplifying leads to
\begin{align*}
  A_x &- \frac{u^{\alpha-1}}{\alpha-1}B_x
	= (\lambda_1-\lambda_3)u^\beta u_{xxx} \\
	&\phantom{xx}{}+ \big((2\lambda_1+2\lambda_2-\lambda_3
	-2\lambda_4)\alpha + (\lambda_1-\lambda_3)\beta
	- 4\lambda_1-2\lambda_2+3\lambda_3+2\lambda_4\big)u^{\beta-1}
	u_{xx}u_x \\
	&\phantom{xx}{}+ \big((\lambda_3+\lambda_4)\alpha^2 
	+ (\lambda_1+\lambda_2-\lambda_3-\lambda_4)\alpha\beta
	- (\lambda_1+\lambda_2+2\lambda_3+\lambda_4)\alpha \\
	&\phantom{xx}{}
	+ (-2\lambda_1-\lambda_2+2\lambda_3+\lambda_4)\beta + 2\lambda_1+\lambda_2
	\big)u^{\beta-2}u_x^3.
\end{align*}
We identify the coefficients with those in the expression \eqref{1.eq} for $J$:
\begin{align*}
  1 &= \lambda_1-\lambda_3, \\
  a &= (2\lambda_1+2\lambda_2-\lambda_3-2\lambda_4)\alpha + (\lambda_1-\lambda_3)\beta
	- 4\lambda_1-2\lambda_2+3\lambda_3+2\lambda_4, \\
	b &= (\lambda_3+\lambda_4)\alpha^2 
	+ (\lambda_1+\lambda_2-\lambda_3-\lambda_4)\alpha\beta
	- (\lambda_1+\lambda_2+2\lambda_3+\lambda_4)\alpha \\
	&\phantom{xx}{}
	+ (-2\lambda_1-\lambda_2+2\lambda_3+\lambda_4)\beta + 2\lambda_1+\lambda_2.
\end{align*}
The general solution of this linear system for 
$(\lambda_1,\lambda_2,\lambda_3,\lambda_4)$, with free parameter $\lambda_4\in\R$, 
gives \eqref{2.lam}.

Next, let $\alpha=0$. The ansatz for $A$ and $B$ becomes
\begin{align*}
  A &= u^\beta\big(\lambda_1 u^2(-u^{-1})_{xx} - \lambda_2 u^3(-u^{-1})_x^2\big), \\
  B &= u^\beta\big(-\lambda_3 u^3(-u^{-1})_{xx} + \lambda_4 u^4(-u^{-1})_x^2\big).
\end{align*}
Then
\begin{align*}
  A_x-vB_x &= (\lambda_1-\lambda_3)u^\beta u_{xxx}
	+ \big((\lambda_1-\lambda_3)\beta-4\lambda_1-2\lambda_2+3\lambda_3+2\lambda_4\big)
	u^{\beta-1}u_{xx}u_x \\
	&\phantom{xx}{}
	+ \big((-2\lambda_1-\lambda_2+2\lambda_3+\lambda_4)\beta+2\lambda_1+\lambda_2\big)
  u^{\beta-2}u_x^3.
\end{align*}
Identifying the coefficients with those in \eqref{1.eq} again gives a linear
system for the parameters $(\lambda_1,\lambda_2,\lambda_3,\lambda_4)$. The general solution
reads as
\begin{align}
  \lambda_1 &= \frac{\beta^2-a\beta+a+2b+3}{\beta+3} - \frac{2}{\beta+3}\lambda_4, 
	\nonumber \\
  \lambda_2 &= -\frac{2a+b+6}{\beta+3} + \frac{\beta+4}{\beta+3}\lambda_4, 
	\label{2.lam0} \\
	\lambda_3 &= \frac{\beta^2-a\beta-\beta+a+2b}{\beta+3} 
	- \frac{2}{\beta+3}\lambda_4, \nonumber
\end{align}
These expressions are the same as \eqref{2.lam} with $\alpha=0$.
\end{proof}

For the following theorem, we recall definition \eqref{1.ent} of $s_\alpha$ and set
$S_\alpha(u)=\int_\T s_\alpha(u)dx$.

\begin{theorem}\label{thm.ent}
Let $u$ be a smooth positive solution to \eqref{1.eq} and let $\alpha\ge 0$. 
If $K(\alpha,\beta)\ge 0$ (see definition \eqref{1.cond}), 
then $S_\alpha$ is a Lyapunov functional, i.e.\ $dS_\alpha/dt\le 0$ for all $t>0$.
If $K(\alpha,\beta)>0$ then there exists $c(\alpha,\beta)>0$ such that for all $t>0$,
$$
  \frac{dS_\alpha}{dt} + c(\alpha,\beta)\int_\T u^{\alpha+\beta}
	\big(u^{2(1-\alpha)}(u^{\alpha-1})_{xx}^2 + u^{4(1-\alpha)}(u^{\alpha-1})_x^4\big)dx
	\le 0
$$
and there exists another constant $C(\alpha,\beta)>0$ such that for all $t>0$,
$$
  \frac{dS_\alpha}{dt} + C(\alpha,\beta)\int_\T u^{\alpha+\beta}\big(u^{-2}u_{xx}^2
	+ u^{-4}u_x^4\big)dx \le 0.
$$
\end{theorem}

\begin{proof}
Let first $\alpha\ge 0$ with $\alpha\neq 1$. 
We calculate the time derivative of the entropy, using integration by parts twice:
\begin{align}
  \frac{dS_\alpha}{dt} &= \int_\T s_\alpha'(u)\pa_t u dx
	= -\int_\T vJ_x dx = \int_\T Jv_x dx \nonumber \\
  &= \int_\T(A_x - vB_x)v_x dx
	= -\int_\T\big(Av_{xx} - (vv_{xx}+v_x^2)B\big)dx \nonumber \\
  &= -\int_\T \frac{u^{\alpha+\beta}}{(\alpha-1)^2}\bigg[
	(\lambda_1-\lambda_3)\bigg(\frac{v_{xx}}{v}\bigg)^2
	+ (\lambda_2-\lambda_3-\lambda_4)\frac{v_{xx}}{v}\bigg(\frac{v_x}{v}\bigg)^2
	- \lambda_4\bigg(\frac{v_x}{v}\bigg)^4\bigg]dx, \nonumber \\
	&= -\int_\T \frac{u^{\alpha+\beta}}{(\alpha-1)^2}
	P\bigg(\frac{v_{x}}{v},\frac{v_{xx}}{v}\bigg)dx, \label{2.dSdt}
\end{align}
where 
\begin{equation}\label{poly}
  P(\xi_1,\xi_2) = (\lambda_1-\lambda_3)\xi_2^2+ 
  (\lambda_2-\lambda_3-\lambda_4)\xi_2\xi_1 - \lambda_4\xi_1^4.
\end{equation}
The right-hand side of \eqref{2.dSdt} is nonpositive if 
$P(\xi_1,\xi_2)\ge 0$ for all $(\xi_1,\xi_2)\in\R^2$. Taking into account that
$\lambda_1-\lambda_3=1$, this is the case if and only if
$$
  -4\lambda_4 - (\lambda_2-\lambda_3-\lambda_4)^2 \ge 0.
$$
In view of definition \eqref{2.lam} of $\lambda_2$ and $\lambda_3$, we may
interpret $\lambda_4$ as a free parameter und $\lambda_2=\lambda_2(\lambda_4)$
and $\lambda_3=\lambda_3(\lambda_4)$ as affine functions in $\lambda_4$. Therefore,
we require that
\begin{equation}\label{pos}
  f(\lambda_4) := -4\lambda_4 
	- \big(\lambda_2(\lambda_4)-\lambda_3(\lambda_4)-\lambda_4\big)^2 \ge 0.
\end{equation}
We choose the optimal value of $\lambda_4$ by computing the critical value of $f$:
\begin{align*}
  0 = f'(\lambda_4) &= \frac{2(-2\alpha^2 + (-3a+8\beta+3)\alpha - 3a\beta + \beta^2 
	+ 9(a+b) - 15\beta)}{(\beta-2\alpha+3)^2} \\
	&\phantom{xx}{}- \frac{18(\alpha-1)^2}{(\beta-2\alpha+3)^2}\lambda_4.
\end{align*}
This yields
\begin{equation}\label{2.lam4}
  \lambda_4 = \frac{-2\alpha^2 + (-3a+8\beta+3)\alpha - 3a\beta + \beta^2 + 9(a+b) 
	- 15\beta}{9(\alpha-1)^2}.
\end{equation}
Inserting $\lambda_4$ in \eqref{pos} leads to
$$
  0 \le f(\lambda_4) = \frac{4}{9(\alpha-1)^2}\big(-2\alpha^2+(3a-4\beta+9)\alpha
	-2\beta^2+(3a+9)\beta-9(a+b+1)\big),
$$
which is equivalent to $K(\alpha,\beta)\ge 0$, see \eqref{1.cond}.

If $K(\alpha,\beta)>0$, there exists $c_0(\alpha,\beta)>0$ such that for all 
$(\xi_1,\xi_2)\in\R^2$,
$P(\xi_1,\xi_2)\ge c_0(\alpha,\beta)(\xi_1^2+\xi_1^4)$. Inserting this information in 
\eqref{2.dSdt}, we infer that
\begin{align*}
  \frac{dS_\alpha}{dt} &\le -c_0(\alpha,\beta)\int_\T
  \frac{u^{\alpha+\beta}}{(\alpha-1)^2}
	\bigg[\bigg(\frac{v_{xx}}{v}\bigg)^2 + \bigg(\frac{v_x}{v}\bigg)^4\bigg]dx \\
	&= -c_0(\alpha,\beta)\int_\T u^{\alpha+\beta-2}\bigg[\bigg(\frac{u_{xx}}{u}\bigg)^2
	+ 2(\alpha-2)\frac{u_{xx}}{u}\bigg(\frac{u_x}{u}\bigg)^2
	+ (2\alpha^2-6\alpha+5)\bigg(\frac{u_x}{u}\bigg)^4\bigg]dx.
\end{align*}
The discriminant equals
$$
  1\cdot(2\alpha^2-6\alpha+5) - (\alpha-2)^2 = (\alpha-1)^2,
$$
and it is positive for all $\alpha\neq 1$. Therefore, there exists $k(\alpha)>0$ 
such that
$$
  \frac{dS_\alpha}{dt} \le -c_0(\alpha,\beta) k(\alpha)\int_\T u^{\alpha+\beta}
	\bigg(\bigg(\frac{u_{xx}}{u}\bigg)^2 + \bigg(\frac{u_x}{u}\bigg)^4\bigg)dx,
$$
and this gives the conclusion for $c(\alpha,\beta)=c_0(\alpha,\beta)/(\alpha-1)^2$
and $C(\alpha,\beta)=c_0(\alpha,\beta)k(\alpha)$ when $\alpha\neq 1$.

It remains to analyze the case $\alpha=1$. Here, we cannot formulate the flux
as $J=A_x-vB_x$ with $v=s_1'(u)=\log u$, since $J$ does not contain logarithmic terms.
Our idea is to write $J=A_x-wB_x$ with $w=-u^{-1}$ and functions $A$ and $B$
that depend on $w$, $w_x$, and $w_{xx}$. The time derivative of the
entropy $S_1$ can be written in terms of $w$ and its derivatives only, since
the logarithmic term $v=\log u$ only appears with its derivatives.

The formulation $J=A_x-wB_x$ corresponds to the expression used for $\alpha=0$.
In fact, we have
\begin{align*}
  A &= u^{\beta}\big(\lambda_1u^2(-u^{-1})_{xx} - \lambda_2u^3(-u^{-1})_x^2\big)
	= \frac{u^{\beta}}{w}\bigg(\lambda_1 \frac{w_{xx}}{w} 
	+ \lambda_2 \bigg(\frac{w_x}{w}\bigg)^2\bigg), \\
	B &= u^{\beta}\big(-\lambda_3u^3(-u^{-1})_{xx} + \lambda_4u^4(-u^{-1})_x^2\big)
	= \frac{u^{\beta}}{w^2}\bigg(\lambda_3 \frac{w_{xx}}{w} 
	+ \lambda_4 \bigg(\frac{w_x}{w}\bigg)^2\bigg),
\end{align*}
where $\lambda_1,\lambda_2,\lambda_3$ are given by \eqref{2.lam0} and $\lambda_4$
is a free parameter. As before, the time derivative becomes
$$
  \frac{dS_1}{dt} = \int_\T\pa_t u\log u dx = -\int_\T J_x v dx
	= \int_\T Jv_x dx = -\int_\T\big(Av_{xx}-B(v_xw)_x\big)dx.
$$
Set $\xi_1=w_x/w$ and $\xi_2=w_{xx}/w$. 
Since $v_{xx}=-w_{xx}/w+(w_x/w)^2=-\xi_2+\xi_1^2$ and $(v_xw)_x=-w_{xx}=-w\xi_2$, 
we obtain
\begin{align*}
  \frac{dS_1}{dt} &= -\int_\T \frac{u^\beta}{w}\big((\lambda_1\xi_2+\lambda_2\xi_1^2)
	(-\xi_2+\xi_1^2) + (\lambda_3\xi_2+\lambda_4\xi_1^2)\xi_2\big)dx \\
  &= -\int_\T u^{\beta+1}\big(-(\lambda_1\xi_2+\lambda_2\xi_1^2)(-\xi_2+\xi_1^2)
	- (\lambda_3\xi_2+\lambda_4\xi_1^2)\xi_2\big)dx \\
	&= -\int_\T u^{\beta+1} P_1(\xi_1,\xi_2) dx,
\end{align*}
where
\begin{equation}\label{poly2}
  P_1(\xi_1,\xi_2) = (\lambda_1-\lambda_3)\xi_2^2 + (-\lambda_1+\lambda_2-\lambda_4)
	\xi_2\xi_1^2 - \lambda_2\xi_1^4.
\end{equation}
Observe that $\lambda_1-\lambda_3=1$. Thus, $dS_1/dt\ge 0$ if the polynomial
$P_1$ is nonnegative for all $(\xi_1,\xi_2)\in\R^2$, which is the case if and only if
$$
  f_1(\lambda_4) := -4\lambda_2(\lambda_4) 
	- \big(-\lambda_1(\lambda_4)+\lambda_2(\lambda_4)-\lambda_4)^2 \ge 0.
$$
We choose the optimal value $\lambda_4$ from $f_1'(\lambda_4)=0$, i.e.
$$
  \lambda_4 = \frac19\big(\beta^2 - (3a+14)\beta + 9(a+b) + 3\big).
$$
Then $f(\lambda_4)\ge 0$ is equivalent to $K(1,\beta)\ge 0$.

Finally, if $K(1,\beta)>0$, we have $P_1(\xi_1,\xi_2)\ge c_0(1,\beta)(\xi_2^2+\xi_1^4)$
and consequently,
\begin{align*}
  \frac{dS_1}{dt} &\le -c_0(1,\beta)\int_\T u^{\beta+1}\bigg(
	u^2\bigg(\frac{1}{u}\bigg)_{xx}^2 + u^4\bigg(\frac{1}{u}\bigg)_x^4\bigg)dx \\
	&= -c_0(1,\beta)\int_\T u^{\beta+1}
	\big(u^{-2}u_{xx}^2 - 4u^{-3}u_{xx}u_x^2 + 5u^{-4}u_x^4\big)dx.
\end{align*}
The discriminant is positive and there exists $k(1)>0$ such that
$\xi_2^2-4\xi_2\xi_1^2+5\xi_1^4\ge k(1)(\xi_2^2+\xi_1^4)$. We infer that
$$
  \frac{dS_1}{dt} \le -c_0(1,\beta)k(1)\int_\T u^{\beta+1}
	\big(u^{-2}u_{xx}^2 + u^{-4}u_x^4\big)dx,
$$
which concludes the proof with $c(1,\beta)=c_0(1,\beta)$ and
$C(1,\beta)=c_0(1,\beta)k(1)$.
\end{proof}

\begin{remark}[Examples]\rm
The DLSS equation corresponds to \eqref{1.eq} if $a=-2$, $b=0$, and $\beta=0$.
Then condition \eqref{1.cond} becomes $0<\alpha\le \frac32$,
which is the optimal interval. Choosing $a=b=0$, we obtain the
thin-film equation, and condition \eqref{1.cond} is equivalent to
$\frac32\le \alpha+\beta\le 3$, which again is the optimal parameter range.
\qed
\end{remark}

\begin{remark}[Systematic integration by parts]\rm
The result of Theorem \ref{thm.ent} can be also derived by the method of 
systematic integration by parts of \cite{JuMa06}. 
Our proof is taylored in such a way that it can
be directly ``translated'' to the discrete level. Indeed, the method of \cite{JuMa06}
needs several chain rules that are not available on the discrete level and
our technique needs to be used.

Still, systematic integration by parts and our strategy are strongly related.
Systematic integration by parts means that we are adding so-called
integration-by-parts formulas whose derivatives vanishe. In this way, we
can derive \eqref{2.dSdt},
$$
  \frac{dS_\alpha}{dt}
  = \int_\T(A_x-vB_x)v_x dx + c_1\int_\T(Av_x-Bvv_x)_x dx
  = -\int_\T(Av_{xx}-B(vv_x)_x)dx,
$$
by choosing $c_1=-1$. In the method of systematic integration by parts, we
are adding a term of the type $c_2\int_\T(u^{\alpha+\beta}(v_x/v)^3)_x dx$
and optimize $c_2$. By contrast, the constant $c_1$ is fixed, but we optimize
$\lambda_4$ in the formulation of $A$ and $B$. In both cases, just one parameter
needs to be optimized.
\qed
\end{remark}


\section{General discretized equation}\label{sec.disc}

We ``translate'' the computations of the previous section to the discrete level.
For this, we use the discrete entropy
\begin{equation}\label{3.ent}
  S_\alpha^h(u) = h\sum_{i\in\T_h}s_\alpha(u_i),
	\quad\mbox{where } \alpha\ge 0,\ \alpha\neq 1,
\end{equation}
where $u_i=u(i)$ for $i\in\T_h$ and $s_\alpha$ is defined in \eqref{1.ent}.
We recall scheme \eqref{1.d1}:
\begin{equation}\label{3.sch}
  \pa_t u_i = -\frac{1}{h}(J_{i+1/2}-J_{i-1/2}), \quad
	J_{i+1/2} = \frac{1}{h}(A_{i+1}-A_i) - \frac{1}{2h}(v_{i+1}+v_i)(B_{i+1}-B_i),
\end{equation}
where the functions $A_i$ and $B_i$ are given by
\begin{equation}\label{3.AB}
  A_i = \frac{\bar u_i^{\alpha+\beta}}{(\alpha-1)^2v_i}(\lambda_1\xi_{2,i}
	+\lambda_2(\alpha-1)\xi_{1,i}^2), \quad
	B_i = \frac{\bar u_i^{\alpha+\beta}}{(\alpha-1)^2v_i^2}(\lambda_3\xi_{2,i}
	+\lambda_4(\alpha-1)\xi_{1,i}^2),
\end{equation}
$\bar u_i$ is an arbitrary average of $u$ around the point $ih$,
and $\xi_{2,i}$ and $\xi_{1,i}^2$ are discrete analogs of
$v_{xx}/v$ and $(v_x/v)^2$:
\begin{equation}\label{3.xi21}
 	\xi_{2,i} := \frac{1}{v_ih^2}(v_{i+1}-2v_i+v_{i-1}), \quad
  \xi_{1,i}^2 := \frac{1}{2v_i^2h^2}\big((v_{i+1}-v_i)^2+(v_i-v_{i-1})^2\big).
\end{equation}
The parameters $\lambda_i$ for $i=1,2,3,4$ are given by \eqref{2.lam} and 
\eqref{2.lam4}. For later use, we note that
\begin{align*}
  \xi_{1,i}^2 &= \frac{1}{2v_i^2h^2}\big((v_{i+1}^2-2v_i^2+v_{i-1}^2) 
	- 2v_i(v_{i+1}-2v_i+v_{i-1})\big) \\
	&= \frac{1}{2v_i^2h^2}(v_{i+1}^2-2v_i^2+v_{i-1}^2) - \xi_{2,i}
\end{align*}
implies that
\begin{equation}\label{3.xi}
  \frac{1}{2h^2}(v_{i+1}^2-2v_i^2+v_{i-1}^2)
	= v_i^2(\xi_{2,i}+\xi_{1,i}^2), \quad
	\frac{1}{h^2}(v_{i+1}-2v_i+v_{i-1}) = v_i\xi_{2,i},
\end{equation}

Recall definition \eqref{3.ent} of $S_\alpha^h$ and condition \eqref{1.cond} for
$K(\alpha,\beta)$.

\begin{theorem}\label{thm.ent2}
Let $(u_i)_{i\in\T_h}$ be a positive solution to \eqref{3.sch}--\eqref{3.xi21}
and let $\alpha\ge 0$, $\alpha\neq 1$.
If $K(\alpha,\beta)\ge 0$ then $dS_\alpha^h/dt\le 0$ for all $t>0$. 
Moreover, if $K(\alpha,\beta)>0$,
$$
  \frac{dS_\alpha^h}{dt} + c(\alpha,\beta)h\sum_{i\in\T_h}
	\bar u_i^{\alpha+\beta}
	\big(\xi_{2,i}^2 + \xi_{1,i}^4\big) \le 0
$$
with the same constant $c(\alpha,\beta)>0$ as in the proof of Theorem \ref{thm.ent}.
\end{theorem}

\begin{proof}
Let $\alpha>0$, $\alpha\neq 1$.
We compute the time derivative of the discrete entropy, using summation by parts twice:
\begin{align*}
  \frac{dS_\alpha^h}{dt} &= h\sum_{i\in\T_h}s_\alpha'(u)\pa_t u_i
	= -\sum_{i\in\T_h}v_i(J_{i+1/2}-J_{i-1/2})
  = \sum_{i\in\T_h}(v_{i+1}-v_i)J_{i+1/2} \\
	&= \frac{1}{h}\sum_{i\in\T_h}(v_{i+1}-v_i)
	\bigg((A_{i+1}-A_i)-\frac12(v_{i+1}+v_i)(B_{i+1}-B_i)\bigg) \\
	&= \frac{1}{h}\sum_{i\in\T_h}\bigg((v_{i+1}-v_i)(A_{i+1}-A_i)
	- \frac12(v_{i+1}^2-v_i^2)(B_{i+1}-B_i)\bigg) \\
	&= -\frac{1}{h}\sum_{i\in\T_h}\bigg((v_{i+1}-2v_i+v_{i-1})A_i
	- \frac12(v_{i+1}^2-2v_i^2+v_{i-1}^2)B_i\bigg).
\end{align*}
In the last step, we recognize the discrete analog of the chain rule
$(v_xv)_x=\frac12(v^2)_{xx}$. Inserting \eqref{3.xi}, we find that
\begin{align*}
  \frac{dS_\alpha^h}{dt}
	&= -h\sum_{i\in\T_h}\big(\xi_{2,i}v_iA_i - (\xi_{2,i}+\xi_{1,i}^2)v_i^2B_i\big) \\
	&= -h\sum_{i\in\T_h}\frac{\bar u_i^{\alpha+\beta}}{(\alpha-1)^2}
	\big((\lambda_1-\lambda_3)\xi_{2,i}^2 
	+ (\lambda_2-\lambda_3-\lambda_4)\xi_{2,i}\xi_{1,i}^2 - \lambda_4\xi_{1,i}^4\big) \\
	&= -h\sum_{i\in\T_h}\frac{\bar u_i^{\alpha+\beta}}{(\alpha-1)^2}
	P(\xi_{1,i},\xi_{2,i}),
\end{align*}
where $P$ is the same polynomial as in \eqref{poly}.
The proof of Theorem \ref{thm.ent} shows that $P(\xi_{1,i},\xi_{2,i})\ge 0$ if
$K(\alpha,\beta)\ge 0$. Moreover, if the strict inequality $K(\alpha,\beta)>0$ holds,
$P(\xi_{1,i},\xi_{2,i})\ge c_0(\alpha,\beta)(\xi_{2,i}^2+\xi_{1,i}^4)$ with the 
same constant as in Theorem \ref{thm.ent}, which translates into the inequality
$$
  \frac{dS_\alpha}{dt} \le -c_0(\alpha,\beta)\sum_{i\in\T_h}
	\frac{\bar u_i^{\alpha+\beta}}{(\alpha-1)^2}
	\big(\xi_{2,i}^2 + \xi_{1,i}^4\big),
$$
finishing the proof. 
\end{proof}

In Theorem \ref{thm.ent}, the existence of positive solutions is assumed.
We show that such solutions exist globally, at least in case $\alpha=0$.

\begin{proposition}\label{prop.ex}
Let $\alpha=0$ and $u_i^0>0$ for $i\in\T_h$. 
Then there exists a global solution $(u_i)_{i\in\T_h}$ to scheme
\eqref{3.sch}--\eqref{3.xi21} and $u_i(0)=u_i^0$ for $i\in\T_h$
and constants $\kappa_1\ge\kappa_0>0$ such that
$$
  0 < \kappa_0 \le u_i(t)\le \kappa_1\quad\mbox{for all }i\in\T_h,\ t>0.
$$
\end{proposition}

\begin{proof}
Scheme \eqref{3.sch} is a system of ordinary differential equations.
According to the Picard--Lindel\"of theorem, there exists a unique local positive
differentiable solution $(u_i)_{i\in\T_h}$ on the maximal time interval
$[0,T)$ for some $T>0$. This solution can be extended to $[0,\infty)$ if the
functions $u_i(t)$ are uniformly positive and bounded. By Theorem \ref{thm.ent2},
$$
  h\sum_{i\in\T_h}(-\log u_i(t))\le h\sum_{i\in\T_h}(-\log u_i^0),
$$
Moreover, scheme \eqref{3.sch} conserves the mass, $h\sum_{i\in\T_h}u_i(t)
=h\sum_{i\in\T_h}u_i^0$. This shows that
$$
  h\sum_{i\in\T_h}(u_i(t)-\log u_i(t))\le h\sum_{i\in\T_h}(u_i^0-\log u_i^0).
$$
Since $s\mapsto s-\log s$ diverges for $s\to 0$ and $s\to \infty$, there exist
constants $\kappa_0>0$ and $\kappa_1>0$ such that $\kappa_0\le u_i(t)\le \kappa_1$
for all $i\in\T_h$ and $t>0$. Therefore, we can extend the solution globally.
\end{proof}

\begin{remark}[Noncentral scheme for $J_{i+1/2}$]\label{rem.xi}\rm
	A more direct discrete analog of $v_{xx}/v$ and $v_x/v$ is given by
	$$
	\xi_{2,i} = \frac{1}{v_ih^2}(v_{i+1}-2v_i+v_{i-1}), \quad
	\xi_{1,i}^2 = \frac{1}{v_i^2h^2}(v_i-v_{i-1})^2.
	$$
	In this situation, the scheme for $J$ needs to be noncentral,
	$$
	J_{i+1/2} = \frac{1}{h}(A_{i+1}-A_i) - \frac{1}{h}v_i(B_{i+1}-B_i),
	$$
	where $A_i$ and $B_i$ are as before. 
	Indeed, it follows from summation by parts for $\alpha\ge 0$, $\alpha\neq 1$ that
	\begin{align*}
	\frac{dS_\alpha}{dt}
	&= \frac{1}{h}\sum_{i\in\T_h}(v_{i+1}-v_i)
	\big((A_{i+1}-A_i) - v_i(B_{i+1}-B_i)\big) \\
	&= -\frac{1}{h}\sum_{i\in\T_h}\big[(v_{i+1}-2v_i+v_{i-1})A_i
	- \big((v_{i+1}-v_{i})v_{i} - (v_{i}-v_{i-1})v_{i-1}\big)B_i\big] \\
	&= -\frac{1}{h}\sum_{i\in\T_h}\big[(v_{i+1}-2v_i+v_{i-1})A_i
	- \big(v_i(v_{i+1}-2v_i+v_{i-1}) + (v_i-v_{i-1})^2\big)B_i\big] \\
	&= -h\sum_{i\in\T_h}\big(\xi_{2,i}v_iA_i - (\xi_{2,i}+\xi_{1,i}^2)v_i^2B_i\big) 
	= -h\sum_{i\in\T_h}\bar u_i^{\alpha+\beta}P(\xi_{1,i},\xi_{2,i}),
	\end{align*}
	and we can conclude as before.
	\qed
\end{remark}

The case $\alpha=1$ has to be treated in a slightly different way. 
Since $\lim_{\alpha\to 1}(\alpha-1)v_i=\lim_{\alpha\to 1} u_i^{\alpha-1}=1$,
we consider scheme
\begin{equation}\label{3.sch1}
\pa_t u_i = -\frac{1}{h}(J_{i+1/2}-J_{i-1/2}), \quad
J_{i+1/2} = \frac{1}{h}(A_{i+1}-A_i) - \frac{1}{2h}(w_{i+1}+w_i)(B_{i+1}-B_i),
\end{equation}
with $A_i$ and $B_i$ defined via
\begin{equation}\label{3.AB1}
  A_i = -\bar u_i^{\beta+1}(\lambda_1\xi_{2,i}-\lambda_2\xi_{1,i}^2), \quad
	B_i = \bar u_i^{\beta+1}(\lambda_3\xi_{2,i}-\lambda_4\xi_{1,i}^2), 
\end{equation}
and 
\begin{align}
  \xi_{2,i} &= \frac{1}{2h^2}\big((v_{i+1}-v_i)(w_{i+1}+w_i)-(v_i-v_{i-1})(w_i+w_{i-1})
	\big), \nonumber \\
  \xi_{1,i}^2 &= \xi_{2,i} + \frac{1}{h^2}(v_{i+1}-2v_i+v_{i-1}), \label{3.xi1}
\end{align}
where $v_i=\log u_i$ and $w_i=-u_i^{-1}$. The parameters $\lambda_1,\ldots,\lambda_4$
are given by \eqref{2.lam} and \eqref{2.lam4}.

\begin{proposition}\label{prop.ent1}
Let $\alpha=1$ and let $(u_i)_{i\in\T_h}$ be a positive solution to \eqref{3.sch1}, 
\eqref{3.AB1}, and \eqref{3.xi1}.
If $K(1,\beta)\ge 0$ then $dS_1^h/dt\le 0$ for all $t>0$. Moreover, 
if $K(1,\beta)>0$,
$$
  \frac{dS_1^h}{dt} + c(1,\beta)h\sum_{i\in\T_h}
	\bar u_i^{\beta+1}\big(\xi_{2,i}^2 + \xi_{1,i}^4\big) \le 0,
$$
with the same constant $c(1,\beta)>0$ as in the proof of Theorem \ref{thm.ent}.
\end{proposition}

\begin{proof}
We have with $v_i=\log u_i$ and $w_i=-u_i^{-1}$:
\begin{align*}
  \frac{dS_1^h}{dt} &= \sum_{i\in\T_h}(v_{i+1}-v_i)J_{i+1/2} \\
	&= h^{-1}\sum_{i\in\T_h}(v_{i+1}-v_i)\bigg((A_{i+1}-A_i)-\frac12(w_{i+1}+w_i)
	(B_{i+1}-B_i)\bigg) \\
	&= -h^{-1}\sum_{i\in\T_h}\bigg((v_{i+1}-2v_i+v_{i-1})A_i \\
	&\phantom{xx}{}
	- \frac12\big((v_{i+1}-v_{i})(w_{i+1}+w_{i}) - (v_{i}-v_{i-1})(w_{i}+w_{i-1})\big)B_i\bigg).
\end{align*}
By definition of $\xi_{1,i}$ and $\xi_{2,i}$,
\begin{align*}
  \frac{1}{h^2}(v_{i+1}-2v_i+v_{i-1}) &= -\xi_{2,i}+\xi_{1,i}^2, \\
  -\frac{1}{2h^2}\big((v_{i+1}-v_{i})(w_{i+1}+w_{i}) - (v_{i}-v_{i-1})(w_{i}+w_{i-1})\big)	&= -\xi_{2,i}.
\end{align*}
Therefore,
\begin{align*}
  \frac{dS_1^h}{dt} &= -h\sum_{i\in\T_h}\bar u_i^{\beta+1}
	\big((\lambda_1\xi_{2,i}-\lambda_2\xi_{1,i}^2)(-\xi_{2,i}+\xi_{1,i}^2)
	+ (\lambda_3\xi_{2,i}-\lambda_4\xi_{1,i})\xi_{2,i}\big) \\
  &= -h\sum_{i\in\T_h}\bar u_i^{\beta+1} P_1(\xi_{1,i},\xi_{2,i}),
\end{align*}
where $P_1$ is the same polynomial as in \eqref{poly2}. It is nonnegative
if and only if $K(1,\beta)\ge 0$ holds.
\end{proof}


\section{Discretized multi-dimensional thin-film equation}\label{sec.tfe}

The ideas of the previous section cannot be adapted in a straightforward way to the
multi-dimensional setting, since there are many possibilities to choose the
finite differences and the discrete variables. One idea is to employ only
{\em scalar} variables like $\Delta u$, $|\na u|^2$, etc., similarly as
for the method of systematic integration by parts of \cite{JuMa06}. Still,
there does not exist a general approach to define the scalar discrete variables,
but we show in this section that the multi-dimensional case can be treated
at least in principle. As an example, we consider the thin-film equation
$$
  \pa_t u = -\diver J, \quad J=u^\beta \na\Delta u\quad\mbox{in }\T^d,
$$
where $\T^d$ is the multi-dimensional torus and $\beta>0$, and the
logarithmic entropy
$$
  S_0(u) = \int_{\T^d}(-\log u)dx.
$$
We show first that we can write $J=\na A-v\na B$, where $v=-u^{-1}$ and $A$, $B$
are functions depending on $\Delta v/v$ and $|\na v|^2/v^2$. 

\begin{lemma}
It holds that $J=\na A-v\na B$, where
\begin{align*}
  A &= u^\beta\big(\lambda_1 u^2\Delta(-u^{-1}) - \lambda_2 u^3|\na(-u^{-1})|^2\big)
	= \frac{u^\beta}{v}\bigg(\lambda_1\frac{\Delta v}{v}
	+ \lambda_2\bigg|\frac{\na v}{v}\bigg|^2\bigg), \\
  B &= u^\beta\big(-\lambda_3 u^3\Delta(-u^{-1}) + \lambda_4 u^4|\na(-u^{-1})|^2\big)
	= \frac{u^\beta}{v^2}\bigg(\lambda_3\frac{\Delta v}{v}
	+ \lambda_4\bigg|\frac{\na v}{v}\bigg|^2\bigg)
\end{align*}
and the parameters $\lambda_i$ are defined by
\begin{equation}\label{4.lam}
  \lambda_1 = \beta+1, \quad \lambda_2 = -2(\beta+1), \quad
  \lambda_3 = \beta, \quad \lambda_4 = -2\beta.
\end{equation}
\end{lemma}

\begin{proof}
We compute
\begin{align*}
  \na A + u^{-1}\na B &= 
	(\lambda_1-\lambda_3)u^\beta\na\Delta u
	+ \big(\beta\lambda_1-(\beta+1)\lambda_3\big)u^{\beta-1}\Delta u\na u \\
	&\phantom{xx}{}
	+ 2\big(-(2\lambda_1+\lambda_2)+2\lambda_3+\lambda_4\big)u^{\beta-1}\na^2 u\na u \\
	&\phantom{xx}{}
	+ \big((1-\beta)(2\lambda_1+\lambda_2)+\beta(2\lambda_3+\lambda_4)\big)
	|\na u|^2\na u.
\end{align*}
Here, $\na^2 u$ denotes the Hessian matrix of $u$.
Identifying the coefficients of $\na A+u^{-1}\na B$ and $J=u^\beta\na\Delta u$
gives the linear system
\begin{align*}
  \lambda_1-\lambda_3 &= 1, &\quad \beta\lambda_1-(\beta+1)\lambda_3 &= 0, \\
  -(2\lambda_1+\lambda_2)+2\lambda_3+\lambda_4 &= 0, &\quad
	(1-\beta)(2\lambda_1+\lambda_2)+\beta(2\lambda_3+\lambda_4) &= 0.
\end{align*}
The unique solution is given by \eqref{4.lam}. 
\end{proof}

\begin{proposition}\label{prop.tfe}
Let $\beta=2$. Then $dS_0/dt\le 0$ for all $t>0$.
\end{proposition}

\begin{proof}
The time derivative of $S_0$ becomes
\begin{align*}
  \frac{dS_0}{dt} &= \int_{\T^d}J\cdot\na v dx
	= \int_{\T^d}(\na A-v\na B)\cdot\na v dx \\
	&= -\int_{\T^d}\big(\Delta v A - (v\Delta v+|\na v|^2)B\big)dx \\
	&= -\int_{\T^d}u^\beta\bigg((\lambda_1-\lambda_3)\bigg(\frac{\Delta v}{v}\bigg)^2
	+ (\lambda_2-\lambda_3-\lambda_4)\frac{\Delta v}{v}\bigg|\frac{\na v}{v}\bigg|^2
	- \lambda_4\bigg|\frac{\na v}{v}\bigg|^4\bigg)dx.
\end{align*}
The polynomial 
\begin{equation}\label{poly0}
  P_0(\xi_1,\xi_2) = (\lambda_1-\lambda_3)\xi_2^2
  + (\lambda_2-\lambda_3-\lambda_4)\xi_2\xi_1^2 - \lambda_4\xi_1^4
\end{equation}
is nonnegative in $\R^2$ if and only if $-4\lambda_4-(\lambda_2-\lambda_3-\lambda_4)^2
=-(\beta-2)^2\ge 0$. Hence, we need to assume that $\beta=2$.
\end{proof}

\begin{remark}[Discussion]\rm
The restriction $\beta=2$ in the previous lemma comes from the
fact that we do not have a free parameter to optimize the inequalities.
One may overcome this issue by allowing $A$ and $B$ to depend on more
variables or by assuming that $A$ and $B$ are matrix-valued with
variables like $\na^2 u$ and $\na u\otimes\na u$ and to formulate
$J=\diver A-v\diver B$. However, this leads to several parameters that need
to be determined and eventually to complicated numerical schemes 
which seem to be less interesting in practice.

Another way to understand the restriction $\beta=2$ is from systematic
integration by parts. Indeed, computing
\begin{align*}
  \frac{dS_0}{dt} 
  &= \int_{\T^d}u^\beta \na\Delta u\cdot\na(-u^{-1})dx
  - \int_{\T^d}\diver(u^{\beta-2}\Delta u\na u)dx \\
  &= -\int_{\T^d}u^{\beta-2}(\Delta u)^2dx 
  - (\beta-2)\int_{\T^d}u^{\beta-3}\Delta u|\na u|^2 dx,
\end{align*}
we see that $S_0$ is a Lyapunov functional if the last term vanishes,
which is the case if $\beta=2$.
This computation suggests the following generalization: 
Let $\beta\in(0,2)$ and consider the R\'enyi entropy $S_\alpha$. Then
$$
  \frac{dS_\alpha}{dt} = -\int_{\T^d}u^{\alpha+\beta-2}(\Delta u)^2 dx
  - (\alpha+\beta-2)\int_{\T^d}u^{\alpha+\beta-3}\Delta u|\na u|^2dx,
$$
and the last term vanishes if $\alpha=2-\beta>0$. 
As for the case $\beta=2$ and $\alpha=0$, which is discussed below, we expect that the
generalization $\beta\in(0,2)$ and $\alpha=2-\beta$ can be ``translated'' to the
discrete case, but we leave the details to the interested reader.
\qed
\end{remark}

We turn to the discrete setting. Let $\T_h^d$ be the discrete multi-dimensional torus,
$u:\T_h^d\to\R$ be a function, and $e_\mu$ be the $\mu$-th unit vector of $\R^d$. 
We write $u_i=u(i)$ for $i\in\T_h^d$ and we introduce the finite differences
$$
  \pa_\mu^+ u_i = \frac{1}{h}(u(i+he_\mu)-u(i)), \quad
	\pa_\mu^- u_i = \frac{1}{h}(u(i)-u(i-he_\mu)), 
$$
where $i\in\T_h^d$ and $\mu=1,\ldots,d$. The discrete divergence of 
$F=(F_1,\ldots,F_d):\T_h^d\to\R^d$ and the discrete gradient and Laplacian 
of $u:\T_h^d\to\R$ are defined by, respectively,
$$
  \diver_h^+ F = \sum_{\mu=1}^d\pa_\mu^+ F_\mu, \quad (\na_h^\pm u)_\mu 
  = \pa_\mu^\pm u, \quad \Delta_h u = \diver_h^+\na_h^- u.
$$
where $\mu=1,\ldots,d$. The discrete analogs of $v_{xx}/v$ and $(v_x/v)^2$ are
$$
  \xi_{2,i} = \frac{\Delta_h v_i}{v_i}, \quad 
	\xi_{1,i}^2 = \bigg|\frac{\na_h^+ v_i}{v_i}\bigg|^2,
$$
where $v_i=-u_i^{-1}$. The numerical scheme reads as
\begin{align}
  & \pa_t u_i = -\diver_h^+ J_i, \quad J_i = \na_h^- A_i - v_i\na_h^- B_i, \nonumber \\
  & A_i = \frac{\bar u_i^{\beta}}{v_i}
	(\lambda_1\xi_{2,i}+\lambda_2\xi_{1,i}^2), \quad
	B_i = \frac{\bar u_i^{\beta}}{v_i^2}
	(\lambda_3\xi_{2,i}+\lambda_4\xi_{1,i}^2), \label{4.sch}
\end{align}
and the coefficients are as in \eqref{4.lam}.

\begin{lemma}
Let $\beta=2$ and $u_i^0>0$ for $i\in\T_h^d$. 
Then there exists a positive solution $(u_i)_{i\in\T_h^d}$ to
\eqref{4.lam} and \eqref{4.sch}, and it holds that $dS_0^h/dt\le 0$ for all $t>0$.
\end{lemma}

\begin{proof}
By the Picard--Lindel\"of theorem, there exists a local smooth positive solution
$(u_i)_{i\in\T_h^d}$. We use the summation-by-parts formula
$$
  \sum_{i\in\T_h^d}v_i\diver_h^+ F_i = -\sum_{i\in\T_h^d}(\na_h^- v_i)\cdot F_i
$$
to compute the time derivative of the entropy:
\begin{align*}
  \frac{dS_0^h}{dt}
    &= h\sum_{i\in\T_h^d}\pa_t u_i v_i
	= -h\sum_{i\in\T_h^d}\diver_h^+ J_i v_i
	= h\sum_{i\in\T_h^d}J_i\cdot\na_h^- v_i \\
	&= h\sum_{i\in\T_h^d}(\na_h^- A_i - v_i\na_h^- B_i)\cdot\na_h^- v_i
	= -h\sum_{i\in\T_h^d}\big(\Delta_h v_iA_i - \diver_h^+(v_i\na_h^- v_i)B_i\big).
\end{align*}
The product rule reads as
\begin{align*}
  \diver_h^+&(v_i\na_h^- v_i) = \sum_{\mu=1}^d\pa_\mu^+(v_i\pa_\mu^- v_i) 
	= \frac{1}{h}\sum_{\mu=1}^d\pa_\mu^+\big(v(i)(v(i)-v(i-he_\mu))\big) \\
	&= \frac{1}{h^2}\sum_{\mu=1}^d\big[v(i+he_\mu)\big(v(i+he_\mu)-v(i)\big)
	- v(i)\big(v(i)-v(i-he_\mu)\big)\big] \\
	&= \frac{1}{h^2}\sum_{\mu=1}^d\big[v(i)\big(v(i+he_\mu)-2v(i)+v(i-he_\mu)\big)
	+ \big(v(i+he_\mu)-v(i)\big)^2\big] \\
  &= v_i\Delta_h v_i + |\na_h^+ v_i|^2.
\end{align*}
Therefore,
\begin{align*}
  \frac{dS_0^h}{dt} &= -h\sum_{i\in\T_h^d}\big(\Delta_h v_iA_i - (v_i\Delta_h v_i 
	+ |\na_h^+ v_i|^2)B_i\big) \\
  &= -h\sum_{i\in\T_h^d}\big(\xi_{2,i}v_iA_i - (\xi_{2,i}+\xi_{1,i}^2)v_i^2B_i\big) \\
	&= -h\sum_{i\in\T_h^d}\big((\lambda_1-\lambda_3)\xi_{2,i}^2
	+ (\lambda_2-\lambda_3-\lambda_4)\xi_{2,i}\xi_{1,i}^2 - \lambda_4\xi_{1,i}^4\big)\\
	&= -h\sum_{i\in\T_h^d}\big(\xi_{2,i}^2 - (\beta+2)\xi_{2,i}\xi_{1,i}^2
	+ 2\beta\xi_{1,i}^4\big).
\end{align*}
This gives the polynomial \eqref{poly0} which is nonnegative in $\R^2$ 
if and only if $\beta=2$. Then $S_0^h(t)=S_0^h(0)$ for all $t>0$, and we
conclude as in the proof of Proposition \ref{prop.ex} that $0<c_0\le u_i(t)\le c_1$
for all $i\in\T_h^d$ and $t>0$, where $c_1\ge c_0>0$ are some constants.
Thus, the solution $u_i(t)$ can be extended to a global one.
\end{proof}


\section{Numerical tests}\label{sec.sim}

We apply our scheme to the thin-film and DLSS equations on the torus
in one and two space dimensions. 
The system of ordinary differential equations is solved by the
command \begin{verbatim} scipy.integrate.solve_ivp \end{verbatim} 
from the SciPy library, which uses the
Backward Differentiation Formula (BDF) method of variable order or the
implicit Runge--Kutta method of the Radau IIA family of order 5.
We used the default values {\verb atol = 1e-3} for the absolute tolerance
and {\verb rtol = 1e-6} for the relative tolerance. The local erroris computed
according to {\verb atol + rtol * abs(u)}.

\subsection{DLSS equation}

The DLSS equation is solved by scheme \eqref{3.sch} using the logarithmic entropy:
\begin{align*}
  & \pa_t u_i = -\frac{1}{h}(J_{i+1/2}-J_{i-1/2}),\quad 
	J_{i+1/2} = \frac{1}{h}(A_{i+1}-A_i) + \frac{1}{2h}(u^{-1}_{i+1}+u^{-1}_i)
	(B_{i+1}-B_i), \\
  & A_i= \bar u_i\bigg(\frac53\xi_{2,i}-\frac73\xi_{1,i}^2\bigg), \quad
  B_i= \bar u_i u_i\bigg(-\frac23\xi_{2,i} + \xi_{1,i}^2\bigg), \\
	& \xi_{1,i}^2 = \frac{\bar{u}_i^2}{2h^2}\big((u^{-1}_{i+1}-u^{-1}_{i})^2
	+(u^{-1}_i-u^{-1}_{i-1})^2\big), \quad 
	\xi_{2,i} = \frac{\bar{u}_i^2}{u_ih^2}(u^{-1}_{i+1}-2u^{-1}_i+u^{-1}_{i-1}),
\end{align*}
where $i\in\T_h$.
Figure \ref{fig.dlss1} shows the solution to the DLSS equation at various
time steps using the initial datum $u^0(x)=\max\{10^{-10},\cos(\pi x)^{16}\}$
and the space grid size $h=1/100$. We see that the solution is not monotone,
since it possesses at $x=0.5$ and $t=10^{-8}$ a local maximum. After some time,
it approaches the constant steady state given by $\int_0^1 u^0(x)dx$.

\begin{figure}[htb]
\includegraphics[width=0.55\linewidth]{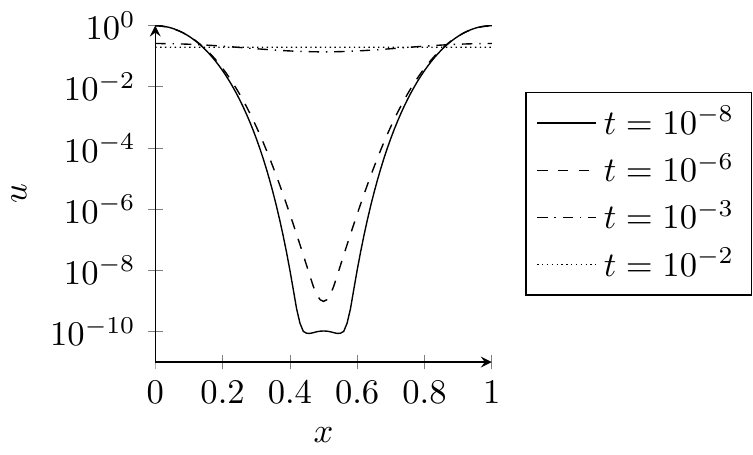}
\caption{Evolution of the DLSS equation in a semi-logarithmic scale, using the
initial datum $u^0(x)=\max\{10^{-10},\cos(\pi x)^{16}\}$.}
\label{fig.dlss1}
\end{figure}

The entropy decay for $\alpha=0$ is illustrated in Figure \ref{fig.dlss2} (left). 
We used the initial datum $u^0(x)=2-10^{-6}$ for $x\in(0,0.5)$ and
$u^0(x)=10^{-6}$ for $x\in(0.5,1)$. We observe in the semi-logarithmic plot that
the decay is exponential, as expected. The rate degrades for
larger times when the $\ell^2$ error dominates, i.e., when the grid is 
rather coarse.

The $\ell^2$ error (in space and time) at time $t=0.001$ is shown in 
Figure \ref{fig.dlss2} (right), using the initial datum $u^0(x)=1+0.5\sin(2\pi x)$
for $x\in(0,1)$. As an explicit solution is not known, we use a numerical
solution with $h=1/2048$ as the reference solution.
As expected, the convergence rate is roughly of second order.

\begin{figure}[htb]
\includegraphics[width=0.45\linewidth]{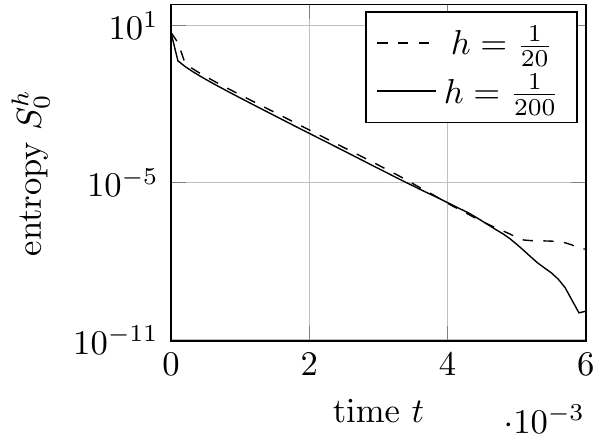}
\includegraphics[width=0.45\linewidth]{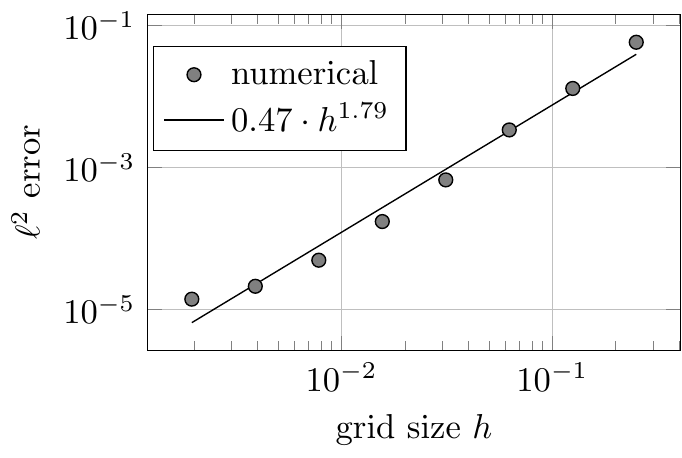}
\caption{Left: Decay of the logarithmic entropy $s_0(u(t))$ for two different
space grid sizes $h=1/20$ and $h=1/200$. 
Right: Convergence of the $\ell^2$ error. The dots are the values from the numerical
solution, the solid line is the regression curve.}
\label{fig.dlss2}
\end{figure}



\subsection{Thin-film equation}

The thin-film equation is solved by scheme \eqref{3.sch} using the logarithmic entropy:
\begin{align*}
  & \pa_t u_i = -\frac{1}{h}(J_{i+1/2}-J_{i-1/2}),\quad 
  J_{i+1/2} = \frac{1}{h}(A_{i+1}-A_i) + \frac{1}{2h}
  (u^{-1}_{i+1}+u^{-1}_i)(B_{i+1}-B_i), \\
  & A_i = u_i^{\beta+1}\bigg(\frac{7\beta+9}{9}\xi_{2,i} 
  + \frac{\beta^2 - 14\beta - 18}{9}\xi_{1,i}^2\bigg), \\
  & B_i = u_i^{\beta+2}\bigg(-\frac{7 \beta}9\xi_{2,i} 
  + \frac{15\beta-\beta^2}{9}\xi_{1,i}^2\bigg), \\
  & \xi_{1,i}^2 = \frac{u_i^2}{2h^2}\big((u^{-1}_{i+1}-u^{-1}_{i})^2
  +(u^{-1}_i-u^{-1}_{i-1})^2\big), 
  \quad \xi_{2,i} = \frac{u_i}{h^2}(u^{-1}_{i+1}-2u^{-1}_i+u^{-1}_{i-1}),
\end{align*}
where $i\in\T_h$.
The solutions at different times, emanating from the initial datum 
$u^0(x)=1+0.5\sin(2\pi x)$,
are shown in Figure \ref{fig.tfe}, where we have chosen $\beta=2$. Again, the
solutions converge to the constant steady state. 
The decay of the logarithmic entropy is illustrated in Figure \ref{fig.entropytfe},
using $\beta=2$ and the initial datum $u^0(x)=1+(1-10^{-16})\sin(2\pi x)$ for 
$x\in(0,1)$. The decay rate is exponential over a large time interval.

\begin{figure}[htb]
	\includegraphics[width=0.35\linewidth]{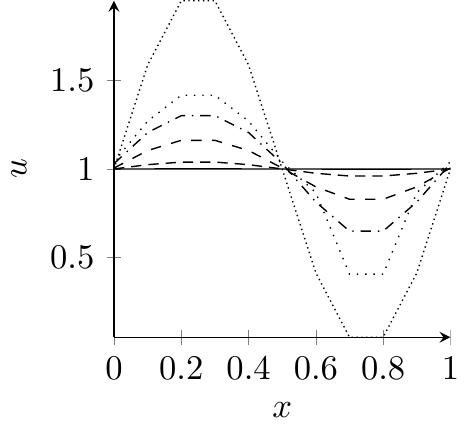}
	\includegraphics[width=0.35\linewidth]{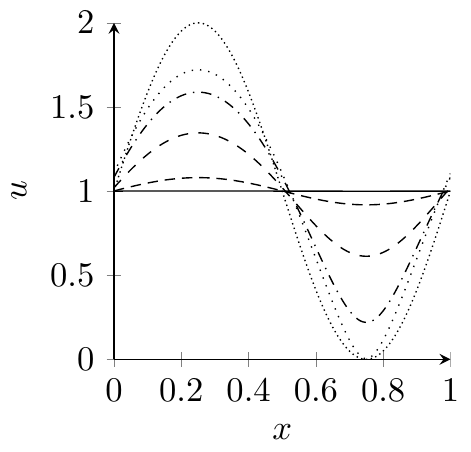}
	\caption{Evolution of the solution to the thin-film equation at times
		$t=0$ (densely dotted), $t=2\cdot 10^{-4}$ (dotted), 
		$t=5\cdot 10^{-4}$ (dash-dotted), 
		$t=1\cdot 10^{-3}$ (dashed), $t=2\cdot 10^{-3}$ (densely dashed), 
		and $t=5\cdot 10^{-3}$ (solid) and grid sizes $h=1/10$ (left), $h=1/200$ (right).}
	\label{fig.tfe}
\end{figure}

\begin{figure}[htb]
	\includegraphics[width=0.5\linewidth]{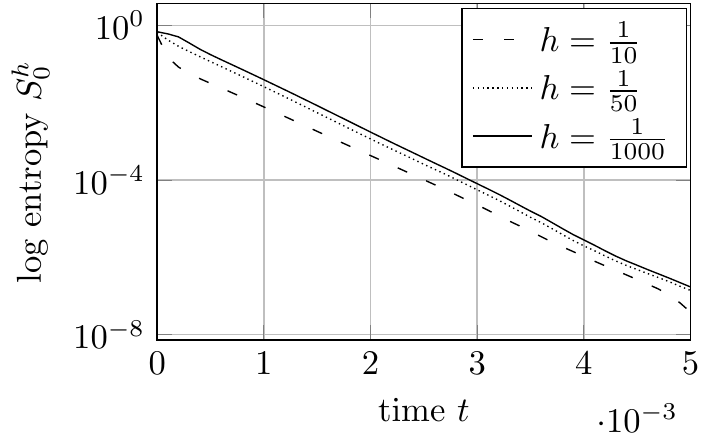}
	\caption{Decay of the logarithmic entropy $S_0(u(t))$ for various
		space grid sizes.}
	\label{fig.entropytfe}
\end{figure}

Finally, we present a numerical example in two space dimensions. As the initial
datum, we choose a lantern picture with  $77\times 100$ pixels in gray scale; 
see Figure \ref{fig.lantern} (top left). 
The evolution of the discrete solution is shown in the remaining
panels of Figure \ref{fig.lantern} for various times. 
The values $u=0$ and $u=1$ correspond to black and white,
respectively. Because of the periodic boundary conditions, we observe a small gray
band at the lower right boundary.
Interestingly, the solution shows a denoising effect, especially
for $t=3\cdot 10^{-8}$. For larger times, the diffusion drives the solution
to the constant steady state. 
These results are not surprising, as fourth-order parabolic equations have been 
suggested in the literature for image denoising. For instance, Bertozzi and Greer 
\cite{BeGr04} analyzed
$$
  \pa_t u = -\diver\big(g((\Delta u)^2)\na\Delta u\big),
$$
where $g$ is a diffusivity function, while Wei \cite{Wei99} considered
$$
  \pa_t u = -\diver\big(g(|\na u|^2)\na\Delta u\big).
$$
This model was generalized to fractional derivatives; see, e.g., \cite{GuLo11}.
An example is the equation
$$
  \pa_t u = -\diver\big(g(-(\Delta)^{1-\eps}u)\na\Delta u\big), \quad \eps>0,
$$
which formally reduces to a general thin-film equation in the limiting case $\eps=1$. 
We do {\em not} claim that the thin-film equation is a good image
denoising model; our numerical example is just a nice illustration.

\begin{figure}[htb]
	\includegraphics[width=0.3\linewidth]{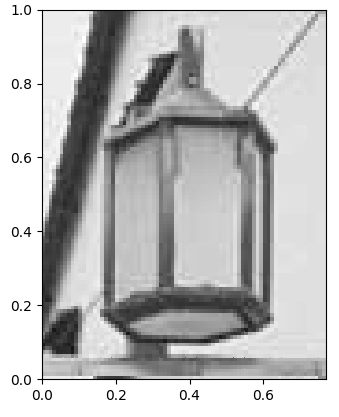}
	\includegraphics[width=0.3\linewidth]{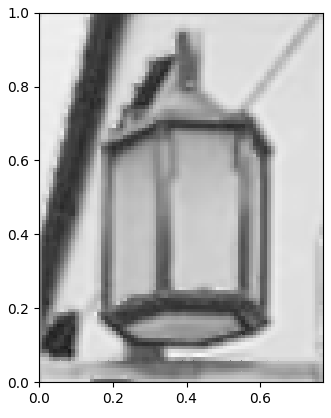} \\
	\includegraphics[width=0.3\linewidth]{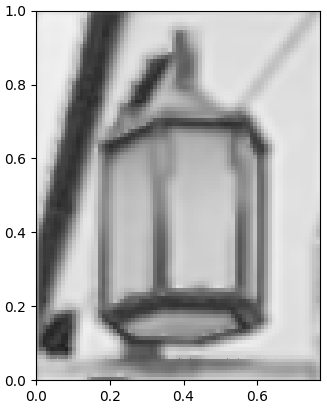}
	\includegraphics[width=0.3\linewidth]{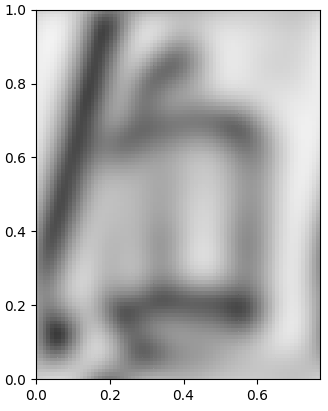}
	\caption{Evolution of the solution to the two-dimensional thin-film
		equation with $\beta=2$, $t=0$ (top left), $t=3\cdot 10^{-9}$ (top right),
		$t=10^{-8}$ (bottom left), $t=10^{-6}$ (bottom right).}
	\label{fig.lantern}
\end{figure}

Finally, we show the entropy decay of the two-dimensional example in
Figure \ref{fig.ent2}. The decay rate is exponential until approximately
$t=10^{-2}$. For later times, the numerical error dominates. Observe, however,
that we obtain denoising for very small times, like $t=10^{-9}\ldots 10^{-8}$,
where the decay rate is still exponential.

\begin{figure}[htb]
	\includegraphics[width=0.5\linewidth]{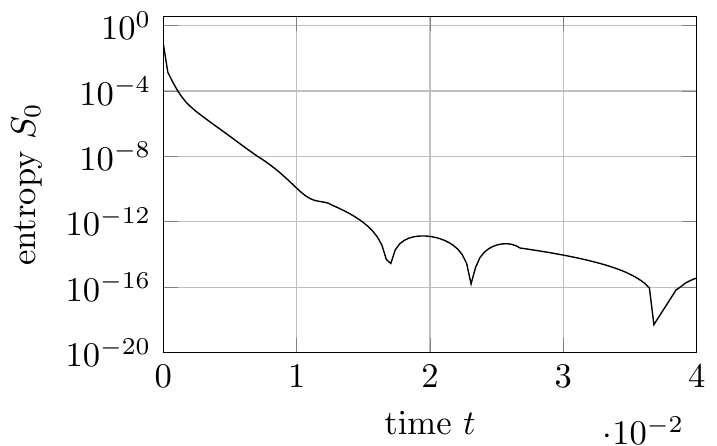}
	\caption{Decay of the logarithmic entropy $S_0(u(t))$ for various
		space grid sizes.}
	\label{fig.ent2}
\end{figure}



\begin{thebibliography}{11}
\bibitem{Ber95} F.~Bernis. Viscous flows, fourth order nonlinear degenerate parabolic 
equations and singular elliptic problems. In: J.I. D\'{\i}az et al. (eds.). 
{\em Free Boundary Problems: Theory and Applications}. Longman Sci. Tech., 
Pitman Res. Notes Math. Ser. 323 (1995), 40--56.

\bibitem{BeGr04} A.~Bertozzi and J.~Greer. Low-curvature image simplifiers: global 
regularity of smooth solutions and Laplacian limiting schemes. 
{\em Commun. Pure Appl. Math.} 57 (2004), 764--790.

\bibitem{BCDP06} A.-S.~Boudou, P.~Caputo, P.~Dai Pra, and G.~Posta. Spectral gap
estimates for interacting particle systems via a Bochner-type identity.
{\em J. Funct. Anal.} 232 (2006), 222--258.

\bibitem{BEJ14} M.~Bukal, E.~Emmrich, and A.~J\"ungel. Entropy-stable and 
entropy-dissipative approximations of a fourth-order quantum diffusion equation. 
{\em Numer. Math.} 127 (2014), 365--396.

\bibitem{CDP09} P.~Caputo, P.~Dai Pra, and G.~Posta. Convex entropy decay via
the Bochner--Bakry--Emery approach. {\em Ann. Inst. H. Poincar\'e Prob. Stat.}
45 (2009), 734--753.

\bibitem{DGG98} R.~Dal Passo, H.~Garcke, and G.~Gr\"un. On a fourth order degenerate 
parabolic equation: global entropy estimates and qualitative behaviour of solutions. 
{\em SIAM J. Math. Anal.} 29 (1998), 321--342.

\bibitem{DLSS91} B.~Derrida, J.~Lebowitz, E.~Speer, and H.~Spohn. Fluctuations of a 
stationary nonequilibrium interface. {\em Phys. Rev. Lett.} 67 (1991), 165--168.

\bibitem{DMM10} B.~D\"uring, D.~Matthes, and J.-P.~Mili\v{s}i\'c. A gradient flow 
scheme for nonlinear fourth order equations. 
{\em Discrete Cont. Dyn. Sys. B} 14 (2010), 935--959.

\bibitem{Egg19} H.~Egger. Structure preserving approximation of dissipative evolution 
problems. {\em Numer. Math.} 143 (2019), 85--106.

\bibitem{FaMa15} M.~Fathi and J.~Maas. Entropic Ricci curvature bounds for
discrete interacting systems. {\em Ann. Appl. Prob.} 26 (2016), 1774--1806.

\bibitem{FuMa10} D.~Furihata and T.~Matsuo. {\em Discrete Variational Derivative 
Method}. Chapman and Hall/CRC Press, Boca Raton, Florida, 2010.

\bibitem{GuLo11} P.~Guidotti and K.~Longo. Well-posedness for a class of fourth 
order diffusions for image processing. 
{\em Nonlin. Diff. Eqs. Appl. NoDEA} 18 (2011), 407--425.

\bibitem{HuLi19} X.~Huo and H.~Liu. A positivity-preserving and energy stable scheme 
for a quantum diffusion equation. Submitted for publication, 2019. arXiv:1912.00813.

\bibitem{JuMa08} A.~J\"ungel and D.~Matthes. The Derrida--Lebowitz--Speer--Spohn 
equation: existence, non-uniqueness, and decay rates of the solutions. 
{\em SIAM J. Math. Anal.} 39 (2008), 1996--2015. 

\bibitem{JuMa06} A.~J\"ungel and D.~Matthes. An algorithmic construction of entropies 
in higher-order nonlinear PDEs. {\em Nonlinearity} 19 (2006), 633--659.

\bibitem{JuYu17} A.~J\"ungel and W.~Yue. Discrete Bochner inequalities via the
Bochner--Bakry--Emery approach for Markov chains. {\em Ann. Appl. Prob.} 27 (2017),
2238--2269.

\bibitem{JuMi15} A.~J\"ungel and J.-P.~Mili\v{c}i\'c. Entropy dissipative one-leg 
multistep time approximations of nonlinear diffusive equations. 
{\em Numer. Meth. Partial Diff. Eqs.} 31 (2015), 1119--1149.

\bibitem{LMS12} S.~Lisini, D.~Matthes, and G.~Savar\'e. Cahn--Hilliard and thin film 
equations with nonlinear mobility as gradient flows in weighted-Wasserstein metrics.
{\em J. Diff. Eqs.} 253 (2012), 814--850.

\bibitem{MaMa16} J.~Maas and D.~Matthes. Long-time behavior of a finite volume 
discretization for a fourth order diffusion equation. 
{\em Nonlinearity} 29 (2016), 1992--2023.

\bibitem{MaOs17} D.~Matthes and H.~Osberger. A convergent Lagrangian discretization 
for a nonlinear fourth-order equation. {\em Found. Comput. Math.} 17 (2017), 73--126.

\bibitem{Wei99} G.~Wei. Generalized Perona--Malik equation for image restoration. 
{\em IEEE Signal Process. Lett.} 6 (1999), 165--167.

\bibitem{ZhBe00} L.~Zhornitskaya and A.~Bertozzi. Positivity-preserving numerical
schemes for lubrication-type equations. {\em SIAM J. Numer. Anal.} 37 (2000), 523--555.

\end{thebibliography}
\end{document}